\documentclass[a4paper, 11pt]{amsart}

\usepackage{amsmath,amssymb,amsthm,amscd,enumerate,setspace}
\usepackage[all]{xy}
\usepackage[dvips]{graphicx}

\usepackage{xcolor}
\usepackage{mdframed}
\newmdenv[backgroundcolor=yellow]{shaded}
\usepackage{hyperref}
\hypersetup{colorlinks=true, linkcolor=red, citecolor=blue}

\input xypic
\xyoption{all}
\newcommand{\llar}{-\kern-5pt-\kern-5pt\longrightarrow}
\newcommand{\lllar}{-\kern-5pt-\kern-5pt\llar}

\newtheorem{Theorem}{Theorem}[section]
\newtheorem{Lemma}[Theorem]{Lemma}

\newtheorem{Proposition}[Theorem]{Proposition}

\newtheorem{Conjecture}[Theorem]{Conjecture}

\newtheorem{Remark}[Theorem]{Remark}
\newtheorem{Example}[Theorem]{Example}
\newtheorem{Definition}[Theorem]{Definition}
\newtheorem{Question}[Theorem]{Question}
\newtheorem{Questions}[Theorem]{Questions}

\newtheorem{Sticky Points}[Theorem]{Sticky Points}

\def\bl{[\![}
\def\br{]\!]}
\def\cdeg{\mbox{\rm cdeg}}

\def\coker{\mbox{\rm coker}}
\def\deg{\mbox{\rm deg}}

\def\ddeg{\mbox{\rm bideg}}

\def\ds{\displaystyle}
\def\e{\mathrm{e}}

\def\Ext{\mbox{\rm Ext}}

\def\h{\mbox{\rm ht}}
\def\Hom{\mbox{\rm Hom}}

\def\ker{\mbox{\rm ker}}
\def\lar{\longrightarrow}
\def\l{{\lambda}}
\def\m{\mathfrak{m}}
\def\n{\mathfrak{n}}
\def\p{{\mathfrak p}}

\def\rank{\mbox{\rm rank}}
\def\rar{\rightarrow}

\def\tdeg{\mbox{\rm tdeg}}

\def\tr{\mbox{\rm tr}}
\def\tratto{\mbox{\rule{3mm}{.2mm}$\;\!$}}

\def\g2{{\mathbf g}}

\def\C{\mathcal{C}}
\def\B{\mathcal{B}}

\def\M{\mathfrak{M}}
\def\P{\mathcal{P}}

\begin{document}

\title{\sc  Generalization of bi-canonical degrees} 
\thanks{AMS 2020 {\em Mathematics Subject Classification}.
Primary 13H15;  Secondary 13H10, 13A99.
The third author was partially supported by the Sabbatical Leave Program at Southern Connecticut State University (Spring 2022). }

\author{J. Brennan}\address{Department of Mathematics, University of Central Florida, 4393 Andromeda Loop N
Orlando, FL 32816, U.S.A.}
\email{Joseph.Brennan@ucf.edu}

\author{L. Ghezzi} \address{Department of Mathematics, New York City College of Technology and the Graduate Center, The City University of New York,
 300 Jay Street, Brooklyn, NY 11201, U.S.A.;
365 Fifth Avenue,
New York, NY 10016, U.S.A.}
 \email{lghezzi@citytech.cuny.edu}

\author{J. Hong}
\address{Department of Mathematics, Southern Connecticut State
University, 501 Crescent Street, New Haven, CT 06515-1533, U.S.A.}
\email{hongj2@southernct.edu}

\author{W. V. Vasconcelos}
\address{Department of Mathematics, Rutgers University, 110
Frelinghuysen Rd, Piscataway, NJ 08854-8019, U.S.A.}

\maketitle

{\em \small Dedicated to Professor Rafael Villareal on the occasion of his birthday for his groundbreaking contributions to Algebra, particularly to Commutative Algebra.}

\begin{abstract}
We discuss invariants of Cohen-Macaulay local rings that admit a canonical module $\omega$. 
 Attached to each such ring $R$,  when $\omega$ is an ideal, there are integers--the type of $R$, the reduction number of $\omega$--that provide valuable metrics to express the deviation of $R$ from being a Gorenstein ring. 
 In \cite{blue1} and \cite{bideg} we enlarged this list with the canonical degree and the bi-canonical degree. In this work we extend the bi-canonical degree to rings where $\omega$ is not necessarily an ideal. We also discuss generalizations to rings without canonical modules but admitting modules sharing some of their properties. 
\end{abstract}

\section{Introduction}

 Let $(R, \m)$ be a Cohen-Macaulay local ring of dimension $d$ that has a canonical module $\omega$. Our central viewpoint is to look
at the properties of $\omega$ as a way to refine our understanding
of $R$.  We recall that $R$ is Gorenstein when $\omega$ is isomorphic to $R$. In \cite{blue1} and \cite{bideg}, if the canonical module $\omega$ is an ideal, we treated metrics aimed at measuring the deviation from $R$ being Gorenstein.
More precisely, two integers arise by considering the following lenghts (when $R$ has dimension one):

\begin{enumerate}

\item If $(a)$ is a minimal reduction of $\omega$, $\cdeg(R) = \lambda(\omega/(a)).$

\item If ${\omega}^{**}$ is the bidual of $\omega$,  $\ddeg(R)=\lambda(\omega^{**}/\omega).$

\end{enumerate}
 
 The canonical degree $\cdeg(R)$ is the focus of  \cite{blue1} and the bi-canonical degree $\ddeg(R)$ is the focus of \cite{bideg}.
They are extended to higher dimensions (see Definition~\ref{defcdeg} and Definition~\ref{bideg}). 

\smallskip

We note that both degrees are invariants of the ring and they vanish if and only if $R$ is Gorenstein in codimension one. 
One of the main points is that the minimal values of these degrees lead to a stratification of the Cohen-Macaulay property, such as almost Gorenstein rings \cite[ 2.4, 2.5, 3.2, 3.3, 6.5]{blue1} and Goto rings \cite[3.1, 4.2, 5.1]{bideg}. For instance, one dimensional almost Gorenstein rings are characterized by having the smallest possible canonical degree.
 
 \smallskip
 An interesting broad question is how to extend these degrees to more general settings.  In this paper we focus on the generalization of the bi-canonical degree (the more suitable degree for computations) to two important cases, outlined below.
 \smallskip
 
First, in Section~\ref{survey} we give a brief survey of the canonical degree and the bi-canonical degree. These are recently introduced degrees but also descendants of a long line of degree functions established by Vasconcelos.  His extensive research on arithmetic degrees (adeg, gdeg, jdet) and cohomological degrees (hdeg, bdeg) demonstrated 
a good understanding of algebraic structures of rings and developed rich applications to normalization of algebras, ideals, and modules. In Section~\ref{notideal}
we extend the bi-canonical degree to a Cohen-Macaulay local ring whose canonical module is not necessarily an ideal. The definition involves the classical theory developed by Auslander \cite{AusBr}.

\begin{Definition}{\rm (Definition~\ref{bidegext}).
Let $(R, \m)$ be a Cohen--Macaulay local ring with a canonical module $\omega$. The {\em bi-canonical degree} of $R$ is 
\[ \ddeg(R) = \deg( \Ext_{R}^1(D(\omega), R) ) + \deg( \Ext_{R}^2(D(\omega), R) ),\]
where ${\ds D(\omega) }$ is the Auslander dual of $\omega$ and $\deg(\tratto)$ is the multiplicity associated with the $\m$--adic filtration.
}\end{Definition}

Note that the condition that the canonical module $\omega$ is an ideal is equivalent to the completion of R being generically Gorenstein. Without the constraint of $\omega$ being an ideal, our new definition and results of bi-canonical degree can be applied to more broad classes of rings.
In Proposition~\ref{torsionless} we discuss conditions under which the canonical module is an ideal, and we see that 
the above definition coincides with the classical one in such case. We also have that $\ddeg(R)=0$ if and only if $R$ is Gorenstein in codimension one (Proposition~\ref{vanishingbideg}). In Proposition~\ref{bidegcomp} we compute the free presentation of the Alexander dual when $R$ is the quotient of a regular local ring. This is helpful in computing bi-canonical degrees in explicit examples.
In Proposition~\ref{trace} we extend a result of \cite{bideg} relating the bi-canonical degree to the trace of the canonical module. We also propose some open questions, inspired by results of \cite{bideg}.

\smallskip
 
In Section~\ref{precanonical} 
we consider Cohen-Macaulay local rings that do not necessarily have a canonical module. 
We look for rings that have ideals with properties similar to those of a canonical ideal. Ideals with similar sets of conditions have appeared in the literature, notably among them closed ideals \cite{BV1} and semidualizing (spherical) modules \cite{V74}.
 We define a class of ideals in this setting.
 
\begin{Definition}{\rm (Definition~\ref{precanonicalideal})
Let $(R, \m)$ be a Cohen-Macaulay local ring. An $R$-ideal $\P$ is called a precanonical ideal if it is closed ($\Hom_{R}(\P, \P)=R$) and $\Ext_{R}^{1}(\P, \P)=0$. 
}
\end{Definition}

We prove conditions under which a precanonical ideal is canonical. 

 \begin{Theorem}{\rm (Theorem~\ref{Ulrich})}
Let $(R, \m)$ be a $1$-dimensional Cohen-Macaulay local ring that has a canonical ideal.
Let $\P$ be an $\m$--primary precanonical ideal of $R$. 
If $c \in \P$ is such that $\P/(c)\simeq (R/\m)^n$, for some $n$, then $\P$ is a canonical ideal.
\end{Theorem}

\smallskip
We define the {\em bi-canonical degree} of $R$ {\em relative to} $\P$ as 
$\ddeg_{\P}(R) = \deg(\P^{**}/\P)$. It is interesting to ask which properties we can obtain from this metric to understand the structures of the ideal and the ring. We end the section by showing in Proposition~\ref{resddeg3} that if  $\deg(\P^{**}/\P)=0$ in a one dimensional ring, then $\P$ is principal.

\section{Canonical and Bi-Canonical Degrees of Rings with Canonical Ideals}\label{survey}

Let $(R, \m)$ be a Cohen-Macaulay local ring of dimension $d$ that has a canonical ideal $\C$. Recall that a canonical ideal is an $R$--ideal isomorphic to a canonical module of $R$.
We look at the properties of $\C$ as a way to refine our understanding of $R$.  
In this section we recall the definition and important properties of the canonical and bicanonical degrees from \cite{blue1} and \cite{bideg}. The canonical degree requires knowledge of the Hilbert coefficients $\e_0(\tratto)$ of $\m$-primary ideals.

\begin{Definition}\label{defcdeg}{\rm \cite[Theorem 2.2, Definition 2.3]{blue1}
Let $(R, \m)$ be a Cohen-Macaulay local ring  of dimension $d \geq 1$ that has a canonical ideal $\C$. The {\em canonical degree of $R$} is the integer
  \[ \cdeg(R)= \sum_{\tiny \h(\p)=1} \cdeg(R_{\p}) \deg(R/\p) = \sum_{\tiny \h(\p)=1} [\e_{0}(\C_{\p}) - \l((R/\C)_{\p})] \deg(R/\p).\]
}\end{Definition}

\begin{Proposition}\label{cdegvan}{\rm \cite[Corollary 2.4]{blue1}}
Let $(R, \m)$ be a Cohen-Macaulay local ring  of dimension $d \geq 1$ that has a canonical ideal $\C$.
Then $\cdeg(R)\geq 0$ and vanishes if and only if  $R$ is Gorenstein in codimension $1$.
\end{Proposition}

If the canonical ideal $\C$ is equimultiple with a minimal reduction $(a)$, then  
\[  \cdeg(R) = \deg(\C/(a)) = \e_0(\m, \C/(a)). \]
We recall that a Cohen-Macaulay local ring $R$ with a canonical module $\omega$  is said to be an {\em almost Gorenstein} ring if there exists an exact sequence of $R$-modules
${\ds 0 \rightarrow R \rightarrow \omega\rightarrow X \rightarrow  0}$
such that $\nu(X)=\e_0(X)$, where $\nu(X)$ is the minimal number of generators of $X$ \cite[Definition 3.3]{GTT15}. 
 
\begin{Proposition}\label{cdegag}{\rm  \cite[Corollary 2.5, Proposition 3.2]{blue1}  }
Let $(R, \m)$ be a Cohen-Macaulay local ring  of dimension $d \geq 1$ that has a canonical ideal $\C$. Suppose that $\C$ is equimultiple.
Let $r$ denote the type of $R$.
\begin{enumerate}[{\rm (1)}]
\item $\cdeg(R) =0$ if and only if $R$ is Gorenstein.
\item $\cdeg(R) \geq r-1$. 
\item If $\cdeg (R) = r -1$, then $R$ is an almost Gorenstein ring.
\end{enumerate}
\end{Proposition}

\begin{Proposition}\label{1dimalmostg}{\rm \cite[Proposition 3.3]{blue1}}
Let $(R,\m)$ be a $1$-dimensional Cohen-Macaulay local ring with a canonical ideal. Let $r$ denote the type of $R$.
 Then $\cdeg(R)= r-1$ if and only if $R$ is an almost Gorenstein ring.
\end{Proposition}

The reduction number of a canonical ideal of $R$ is an invariant of the ring and called the {\em canonical index} of $R$ \cite[Proposition 4.1, Definition 4.2]{blue1}.

\begin{Example}\label{Ex1linkage}{\rm \cite[Example 4.8]{blue1}   Let $A=k[X, Y, Z]$, let
$I=(X^2-YZ, Y^2-XZ, Z^2-XY)$ and $R=A/I$. Let $x, y, z$ be the images of $X, Y, Z$ in $R$. By \cite[Theorem 10.6.5]{ABAbook}, we see that $\C=(x, z)$ is a canonical ideal of $R$ with a minimal reduction $(x)$. 
Then the canonical index is $2$ and $\cdeg(R)=1$. Note that $R$ is almost Gorenstein. }
\end{Example}

Degree formulas are often statements about change of rings.

\begin{Theorem}\label{augmented}{\rm \cite[Theorem 6.8, Corollary 6.9]{blue1} \cite[Theorem 6.5]{GMP11}} Let $(R,\m)$ be a $1$-dimensional Cohen-Macaulay local ring with infinite residue field and a canonical ideal. Suppose $R$ is not a discrete valuation ring.  Let $R \ltimes \m$ denote the idealization of $\m$ over $R$. 
\begin{enumerate}[{\rm (1)}]
\item ${\ds \cdeg(R \ltimes \m)= 2 \, \cdeg(R)+2}$.  
\item $R$ is an almost Gorenstein ring if and only if $R \ltimes \m$ is an almost Gorenstein ring.
\end{enumerate}
\end{Theorem}

Several interesting examples involve the ring $\m: \m$.  

\begin{Theorem}\label{TCdeg}{\rm \cite[Theorem 8.3]{bideg}}
Let $(R, \m)$ be a Cohen Macaulay local ring of dimension $1$ with a canonical ideal.   Let $Q$ be the total ring of fractions of $R$ and $A  = \m:_ Q \m$.  
Suppose $(A, \M)$  is a local ring. Let  $e=[A/ \M : R/\m]$ and $r$ the type of $R$. Then  
\[ \cdeg(A) = e^{-1} ( \cdeg(R) + \e_0(\m) - 2r ).\] 
\end{Theorem}

 \begin{Example}{\rm \cite[Example 8.6]{bideg}
Let $R=k[X_1, \ldots, X_n]/L$ where $L$ is generated by all the quadratic monomials $x_ix_j$ with $i\neq j$. 
Note that $\m = (x_1)\oplus \cdots \oplus (x_n)$ and $\e_0(\m)=n$.  Moreover, the type of $R$ is $r = n-1$.
Since for $i\neq j$ the annihilator of $\Hom((x_i), (x_j))$ has grade positive, by \cite[Remark 8.2]{bideg},  we have
 \[ A =\Hom(\m,\m) = \Hom((x_1), (x_1))\times \cdots \times \Hom((x_n),(x_n))= k[x_1] \times \cdots \times k[x_n].\]
 In particular, $\cdeg(A) =0$.  By Theorem~\ref{TCdeg}, we have
 \[ \cdeg(R) = 2r -\e_0(\m) = 2(n-1)-n = n-2. \]
 Since $\cdeg(R) = r-1$, by Proposition~\ref{1dimalmostg}, $R$ is almost Gorenstein.
 }\end{Example}

\bigskip

The approach in \cite{blue1} is dependent on finding minimal reductions, which is a particularly hard task.
In \cite{bideg},  we chose an approach that seems more amenable to computation in classes of algebras such as monomial subrings, Rees algebras, and numerous classes of rings of dimension $1$.  In the natural embedding
\[ 0 \rar \C \rar \C^{**} \rar \B \rar 0 \]
the module $\B$ remains unchanged when $\C$ is replaced by another canonical module.

\begin{Definition}\label{bideg}{\rm \cite[Theorem 3.1]{bideg}
Let $(R, \m)$ be a Cohen-Macaulay local ring  of dimension $d \geq 1$  that has a canonical ideal $\C$. Then the {\em bi-canonical degree} of $R$ is 
  \[ \ddeg(R)= \deg(\C^{**}/\C) =
     \sum_{\tiny \h( \p)=1} \ddeg(R_{\p}) \deg(R/\p) = \sum_{\tiny \h(\p)=1} [\lambda(R_{\p}/\C_{\p}) - \lambda(R_{\p}/\C^{**}_{\p})] \deg(R/\p).
  \]   
}\end{Definition}

\begin{Proposition}\label{bidegvan}{\rm \cite[Theorem 3.1]{bideg} (Cf. Proposition~\ref{cdegvan})}
Let $(R, \m)$ be a Cohen-Macaulay local ring  of dimension $d \geq 1$  that has a canonical ideal $\C$.  Then $\ddeg(R)\geq 0$ and vanishes if and only if  $R$ is Gorenstein in codimension $1$.
\end{Proposition}

\begin{Proposition}\label{bidegag}{\rm \cite[Theorem 4.2]{bideg} (Cf. Proposition~\ref{1dimalmostg})}
Let $(R, \m)$ be a Cohen-Macaulay local ring of dimension $1$ with a canonical ideal $\C$. If $R$ is almost Gorenstein, then $\ddeg(R) =1$.
\end{Proposition}

The converse of  Proposition~\ref{bidegag} does not hold true, as shown in the following example. 

\begin{Example}{\rm \cite[Example 4.3]{bideg}} {\rm   Consider the monomial ring  $R = \mathbb{Q}[t^5, t^7, t^9]$. 
We have a presentation $R  \simeq \mathbb{Q}[X, Y, Z]/P $, with $P = (Y^2-XZ, X^5-YZ^2, Z^3-X^4 Y)$.  We denote the images of $X, Y, Z$ in $R$   by $x, y, z$ respectively. The type of $R$ is $r=2$.
\begin{enumerate}[(1)]
\item The ideal $\C=(x, y)$ is the canonical ideal of $R$. Since $(x)$ is a minimal reduction of $\C$, we get  $\cdeg(R) = \lambda(\C/(x))  = 2$. 
\item By a Macaulay2 calculation \cite{Macaulay2}, $\C^{**} = (x):  ((x):\C)$ satisfies $\lambda(\C^{**}/\C) = 1$. Thus, $\ddeg(R) = 1$.
\item  By Proposition~\ref{1dimalmostg},  $R$ is not almost Gorenstein.  Hence, $\ddeg(R) = 1$ holds  in a larger class of  rings than almost Gorenstein rings.
\end{enumerate}
}\end{Example}

\begin{Proposition}\label{bidegaug}{\rm \cite[Proposition 5.1]{bideg} (Cf. Theorem~\ref{augmented}) } Let $(R, \m)$ be a $1$-dimensional Cohen-Macaulay local ring with a canonical ideal. Suppose that $R$ is not Gorenstein. Then
\[ \ddeg(R\ltimes \m) = 2\, \ddeg (R) -1.\]
\end{Proposition}

\begin{Proposition}\label{preaGor}{\rm \cite[Proposition 6.1]{bideg} }
Let $R= k[t^a, t^b, t^c]$ be a monomial ring. Let
\[ \varphi = \left[ \begin{array}{ll}
X^{a_1} & Y^{b_2} \\ Y^{b_1} & Z^{c_2} \\ Z^{c_1} & X^{a_2}
\end{array}
\right] \] be the Herzog matrix such that $R \simeq k[X, Y, Z]/I_{2}(\varphi)$. 
Suppose that $a_1\leq a_2$, $b_{2} \leq b_{1}$, and $c_{1} \leq c_{2}$. Then ${\ds \ddeg(R)=  a_{1} b_{2} c_{1} }$. 
\end{Proposition}

By \cite[Theorem 4.1]{GMP11}, the ring $R$ given in Proposition~\ref{preaGor} has the canonical degree $\cdeg(R) = a_{1}b_{1}c_{1}$ or $\cdeg(R) = a_{2}b_{2}c_{2}$. This supports  the following Comparison Conjecture \cite[Conjecture 3.2]{bideg}. 

 \begin{Conjecture}\label{cdegvsfddeg} {\rm 
In general,   $\cdeg(R) \geq \ddeg(R)$. 
}\end{Conjecture}

We close the section with a list of questions about change of rings posed in \cite[6.4 and 6.12]{blue1} and \cite[5.4 and 5.5]{bideg}. 

\begin{Questions}{\rm  Let $(R, \m)$ be a Cohen--Macaulay local ring with a canonical ideal $\C$.

\begin{enumerate}[(1)]

\item Is it true that $\cdeg(R)= \cdeg(R \bl X \br)$? If $\C$ is equimultiple, then it is true. What if $\C$ is not equimultiple?

\item  In \cite[Proposition 6.11]{blue1}  we showed that if $R$ has dimension $d \geq 2$, $\C$ is equimultiple and  $x$ is regular modulo $\C$, then
$\cdeg(R) \leq \cdeg(R/(x)).$  More generally we ask:
when does there exist an $x$ such that $\cdeg(R) = \cdeg(R/(x))$?

\item In accordance with Conjecture~\ref{cdegvsfddeg} we ask: if $\C$ is equimultiple and $x$ is regular modulo $\C$, is $ \ddeg(R) \geq \ddeg(R/(x))$?

\item  How do we relate $\ddeg(A)$ of a graded algebra $A$ to $\ddeg(B)$ of one of its Veronese subalgebras?
If $S$ is a finite injective extension of $R$, is $\ddeg(S) \leq [S:R]  \cdot \ddeg(R)$?

\end{enumerate}
}\end{Questions}

\medskip

\section{Bi-canonical Degrees of Rings with Canonical Modules}\label{notideal}

We propose an extension of the bi-canonical degree to a Cohen--Macaulay local ring $(R, \m)$ whose canonical module $\omega$ is not necessarily an ideal. Let $k$ be the residue field $R/\m$. Recall that a finitely generated $R$--module $\omega$ is called a canonical module of $R$ if
\[ \Ext^{i}_{R} (k, \omega) \simeq \left\{ \begin{array}{ll}
0 \quad & \mbox{if} \;\; i \neq  \dim(R), \vspace{0.1 in} \\
k & \mbox{if} \;\; i = \dim(R).
\end{array}    \right. \]
In particular, if $\dim(R)=0$, then the injective hull of $k$ is the canonical module of $R$ \cite[Proposition 3.1.14, Lemma 3.2.7]{BH}.

\medskip

For an $R$--module $E$, we denote the dual of $E$ by  ${\ds E^{*} = \Hom_{R}(E, R)}$.
By dualizing the finite free presentation of $\omega$,
\[  F_{1} \stackrel{\varphi}{\lar} F_{0} \rar \omega \rar 0,\] 
we obtain the following exact sequence
\[ 0 \rar \omega^{*} \rar F_{0}^{*} \stackrel{\varphi^{*}}{\lar} F_{1}^{*} \rar D(\omega) \rar 0,\]
where ${\ds D(\omega) = \coker(\varphi^{*})}$ is the Auslander dual of $\omega$. 
The module $D(\omega)$ depends on the chosen presentation but the values of ${\ds \Ext_{R}^i(D(\omega), R)}$, for $i\geq 1$, are independent of the presentation. By \cite[Proposition 2.6]{AusBr}, there exists an exact sequence 
\[ 0 \lar \Ext_{R}^1(D(\omega), R) \lar \omega  \lar  \omega^{**} \lar \Ext_{R}^2(D(\omega), R) \lar 0. \]

\begin{Definition}\label{bidegext}{\rm
Let $(R, \m)$ be a Cohen--Macaulay local ring with a canonical module $\omega$. The {\em bi-canonical degree} of $R$ is 
\[ \ddeg(R) = \deg( \Ext_{R}^1(D(\omega), R) ) + \deg( \Ext_{R}^2(D(\omega), R) ),\]
where ${\ds D(\omega) }$ is the Auslander dual of $\omega$ and $\deg(\tratto)$ is the multiplicity associated with $\m$--adic filtration.
}\end{Definition}

An $R$--module $E$ is said to be {\em torsionless} if the natural homomorphism ${\ds \sigma: E \rar E^{**}}$ is injective. Recall that a finite $R$--module $E$ is torsionless if and only if it is a submodule of a finite free module.  The following proposition shows that if $R$ possesses a canonical ideal $\C$, then Definition~ \ref{bideg} and Definition~\ref{bidegext} coincide.

\begin{Proposition}\label{torsionless}
Let $(R, \m)$ be a Cohen--Macaulay local ring with a canonical module $\omega$. Then $\omega$ is torsionless if and only if $\omega$ is isomorphic to an ideal of $R$.
\end{Proposition}

\begin{proof} ($\Leftarrow$) is clear.  ($\Rightarrow$) Suppose $\omega$ is torsionless. Then there exists a finite free $R$--module $F$ and an exact sequence ${\ds 0 \rar \omega \rar F}$.  Let ${\ds \p}$ be a minimal prime of $R$. Then ${\ds \omega_{\p}}$ is the canonical module of ${\ds R_{\p}}$. 
Since ${\ds R_{\p}}$ is Artinian, ${\ds \omega_{\p}}$ is isomorphic to the injective hull ${\ds E(k(\p))}$, where ${\ds k(\p)= R_{\p}/\p R_{\p}}$. In particular, $\omega_{\p}$ is injective. 
Then the exact sequence ${\ds 0 \rar \omega_{\p} \rar F_{\p}}$ splits. Hence ${\ds \omega_{\p}}$ is a finite free $R_{\p}$--module. Moreover, as the injective hull ${\ds E(k(\p))}$,  ${\ds \omega_{\p}}$ is indecomposable. Thus, $\omega_{\p} \simeq R_{\p}$. Hence $R_{\p}$ is Gorenstein. By \cite[Proposition 3.3.18]{BH}, $\omega$ can be identified with an ideal in $R$.
\end{proof}

\begin{Proposition}\label{vanishingbideg}{\rm (Cf. Proposition~\ref{bidegvan})}
Let $(R, \m)$ be a Cohen-Macaulay local ring with a canonical module $\omega$. Then $\ddeg(R)=0$ if and only if $R$ is Gorenstein in codimension $1$.
\end{Proposition}

\begin{proof}  By Definition~\ref{bidegext}, $\ddeg(R)=0$ if and only if $\omega$ is reflexive.  By \cite[Korollar 7.29]{HK2}, this is equivalent to $R_{\p}$ is Gorenstein for all primes $\p$ with height $1$.
\end{proof}

The following proposition shows how the free presentation of the Auslander dual $D(\omega)$ can be explicitly obtained.

\begin{Proposition}\label{bidegcomp}
Let $(S, \n)$ be a regular local ring. Let $I$ be an $S$--ideal of height $g$ such that $R=S/I$ is Cohen-Macaulay. Let 
\[ 0 \rar G_{g} \stackrel{\sigma}{\lar} G_{g-1} \rar \cdots \rar G_{1} \rar G_{0}=S \rar 0,\]
be a minimal free $S$--resolution of $R$.  
\begin{enumerate}[{\rm (1)}]
\item ${\ds  \omega_{R} \simeq \coker(\sigma^{T} \otimes R)}$.
\item ${\ds  D(\omega_{R})  \simeq \coker( \sigma \otimes R) }$.
\item Let ${\ds \varphi= \sigma^{T} \otimes R }$ and ${\ds G_{g}= S^{p}}$. Then $\omega_{R}$ is an ideal if and only if ${\ds \h(I_{p-1}(\varphi) )\geq 1}$ and $I_{p}(\varphi)=0$.
\end{enumerate}
\end{Proposition}

\begin{proof} 
By dualizing with respect to $S$, we obtain a free presentation for the canonical module $\omega_{R}$ of $R$ as an $S$--module \cite[Corollary 3.3.9]{BH}
\[ \Hom_{S}(G_{g-1}, S) \stackrel{\tiny \Hom_{S}(\sigma, S)}{\lar}   \Hom_{S}(G_{g}, S) \rar \omega_{R} \rar 0.\]
By tensoring this with $R=S/I$, we obtain the finite free presentation for $\omega_{R}$ as an $R$--module. That is,
\[ F_{1} \stackrel{\varphi}{\lar} F_{0} \rar \omega_{R} \rar 0,\]
where ${\ds F_{1}= \Hom_{S}(G_{g-1}, S) \otimes_{S} R}$,  ${\ds F_{0}= \Hom_{S}(G_{g}, S) \otimes_{S} R}$, and ${\ds \varphi=\Hom_{S}(\sigma, S) \otimes R}$. In particular,
\[ \omega_{R} \simeq \coker(\Hom_{S}(\sigma, S) \otimes R) = \coker(\sigma^{T} \otimes R). \]
Note that
\[ \varphi^{*} = \Hom_{R}( \varphi, R) = \Hom_{R}(\Hom_{S}(\sigma, S) \otimes R    , R) =  \Hom_{S}(\Hom_{S}(\sigma, S), S) \otimes R = (\sigma^{T})^{T} \otimes R = \sigma \otimes R. \]
Thus, we have
\[ D(\omega_{R}) = \coker(\varphi^{*}) = \coker( \sigma \otimes R).\]
Let ${\ds G_{g}= S^{p}}$. Then $F_{0} = R^{p}$.  By \cite[Proposition 3.3.18]{BH}, $\omega_{R}$ is an ideal if and only if $\omega_{R}$ has rank $1$.  By \cite[Proposition 1.4.3]{BH}, this is equivalent to $\rank(\varphi) = p-1$. By \cite[Proposition 1.4.11]{BH}, it is equivalent to ${\ds \h(I_{p-1}(\varphi) )\geq 1}$ and $I_{p}(\varphi)=0$.
\end{proof}

\begin{Example}\label{ddegexample}{\rm
Let $S$ be a regular local ring with the maximal ideal $(X, Y, Z)$. Let $R=S/I$ be given by the minimal resolution
\[ 0 \lar S^{2} \stackrel{\sigma}{\lar} S^{3} \lar S \lar R \lar 0, \; \mbox{where} \;  \sigma= \left[ \begin{array}{cc}
X^{2} & Y^{2} \\ YZ & ZX \\ XZ & XY 
\end{array}\right]. \]
Let $x, y, z$ be the images of $X, Y, Z$ in $R=S/I$ respectively. As shown in the proof of Proposition~\ref{bidegcomp}, we obtain the free presentation of the canonical module $\omega_{R}$ of $R$:
\[ R^{3} \stackrel{\varphi}{\lar} R^{2} \rar \omega_{R} \rar 0,\]
where ${\ds \varphi = \sigma^{T} \otimes R = \left[\begin{array}{ccc} x^{2} & yz & xz \\ y^{2} & zx & xy  \end{array}  \right] }$.  Since  ${\ds \h(I_{1}(\varphi) ) =0}$, the canonical module $\omega_{R}$ is not an ideal.  Moreover, we have
\[ 0 \rar \omega_{R}^{*} \rar R^{2} \stackrel{\varphi^{*}}{\lar} R^{3} \rar D(\omega_{R}) \rar 0,\]
where ${\ds \varphi^{*} = \sigma \otimes R =  \left[ \begin{array}{cc} x^{2} & y^{2} \\ yz & zx \\ xz & xy 
\end{array}\right]}$.  By using Macaulay2, we compute the bi-canonical degree of $R$:
\[ \ddeg(R) = \deg( \Ext_{R}^1(D(\omega_{R}), R) ) + \deg( \Ext_{R}^2(D(\omega_{R}), R) ) = 1+ 6 = 7. \qedhere \]
}\end{Example}

\bigskip

In \cite{blue1} and \cite{bideg} we studied when minimal values of the degrees are attained. Recall (Proposition~\ref{bidegag}) that if $R$ is an almost Gorenstein ring of dimension $1$, then $\ddeg(R)=1$. The following Remark shows that almost Gorenstein rings possess canonical ideals. 

\begin{Remark}{\rm 
Let $(R, \m)$ be a Cohen-Macaulay ring with a canonical module $\omega$. If there exists an injective homomorphism $\varphi: R \rar \omega$, then $\omega$ is isomorphic to an ideal. In particular, if $R$ is almost Gorenstein then the canonical module is isomorphic to an ideal. 
}\end{Remark}

\begin{proof} Consider the exact sequence ${\ds 0 \rar R \stackrel{\varphi}{\lar} \omega \rar C \rar 0}$, where $C=\coker(\varphi)$. If $C=0$, then $R \simeq \omega$ and we are done. Suppose $C \neq 0$.  Let ${\ds \p}$ be a minimal prime of $R$. Then ${\ds \omega_{\p} \simeq \omega_{R_{\p}} \simeq E(k(\p))}$, where 
${\ds \omega_{R_{\p}} }$ is the canonical module of $R_{\p}$ and $E(k(\p))$ is the injective hull of the residue field of $R_{\p}$. 
Hence by Matlis Duality \cite[Proposition 3.2.12]{BH}, we have ${\ds \l(R_{\p}) = \l(\omega_{\p})}$. Therefore $C_{\p} =0$. This means $R_{\p} \simeq \omega_{R_{\p}}$ for every minimal prime, that is, $R$ is generically Gorenstein. By \cite[Proposition 3.3.18]{BH}, $\omega$ can be identified with an ideal.
\end{proof}

It would be interesting to examine whether Proposition~\ref{bidegag} holds when the canonical module is not necessarily an ideal.

\begin{Question}\label{Gotoring}{\rm 
Let $R$ be a Cohen--Macaulay local ring of dimension $1$ with a canonical module $\omega$. When is $\ddeg(R)=1$?
}\end{Question}

Let ${\ds \Gamma: \omega \otimes \omega^{*} \rar  R }$ be the natural map given by ${\ds \Gamma(x \otimes \alpha) = \alpha(x)}$. The {\em trace}  of $\omega$ is the image of $\Gamma$ in $R$ and is denoted by ${\ds \tr(\omega)}$. The trace degree $\tdeg(R)$ of $R$ was introduced in \cite{HHS19}.
  When the canonical module is isomorphic to an ideal, a standard metric is $\tdeg(R) =\deg(R/\tr(\omega))$.
 In dimension $1$,  $\ddeg(R)$ equals $\tdeg(R)$ \cite[Proposition 2.3]{bideg}.
 A generalization is given by the following proposition.

\begin{Proposition}\label{trace} Let $R$ be a $1$--dimensional Cohen-Macaulay local ring with a canonical module $\omega$. 
\begin{enumerate}[{\rm (1)}]
\item If $\omega$ is torsionless, then  ${\ds \ddeg(R) = \lambda(R/\tr(\omega))}$.
\item If $\omega$ is not torsionless, then  ${\ds \ddeg(R) = \deg \big( \Hom(R/\tr(\omega), \omega) \big) }$. 
\end{enumerate}
\end{Proposition}

\begin{proof} (1) If $\omega$ is torsionless, then $\omega$ is isomorphic to an ideal and the assertion was proved in \cite[Proposition 2.3]{bideg}.

\medskip

\noindent (2)  Let ${\ds \Gamma: \omega \otimes \omega^{*} \rar  R }$ be the natural map and let ${\ds K= \ker(\Gamma)}$.  Consider the exact sequence
\[ 0 \rar K \rar \omega \otimes \omega^{*} \rar \tr(\omega) \rar 0 \]
Set ${\ds E^{\vee} = \Hom(E, \omega)}$. Dualizing the above sequence with respect to $\omega$ we obtain
\[ 0 \rar \big( \tr(\omega) \big)^{\vee} \rar  \big( \omega \otimes \omega^{*}  \big)^{\vee} \rar K^{\vee} =0.\]
Therefore,
\[ \big( \tr(\omega) \big)^{\vee} \simeq \Hom( \omega^{*}\otimes \omega, \omega) \simeq \Hom( \omega^{*}, \Hom(\omega, \omega) )  \simeq \Hom( \omega^{*}, R) = \omega^{**}. \]
By dualizing the sequence ${\ds 0 \rar \tr(\omega) \rar R \rar R/\tr(\omega) \rar 0}$ with respect to $\omega$, we obtain the following.
\[ 0 \rar \big(R/\tr(\omega) \big)^{\vee} \rar \omega \rar \big( \tr(\omega) \big)^{\vee} \simeq \omega^{**}  \rar \Ext^{1}( R/ \tr(\omega), \omega) \rar 0. \]
Therefore, we have
\[ \ddeg(R) = \deg \big( (R/\tr(\omega))^{\vee} \big) + \deg \big( \Ext^{1}( R/ \tr(\omega), \omega)  \big). \]
Since $\omega$ is not isomorphic to an ideal, the height of $\tr(\omega)$ is $0$ \cite[Lemma 2.1]{HHS19}. Therefore, ${\ds \Ext^{1}( R/ \tr(\omega), \omega)  =0}$ and  ${\ds \ddeg(R) = \deg \big( (R/\tr(\omega))^{\vee} \big)  }$.
\end{proof}

We conclude with a question motivated by Proposition~\ref{bidegaug}.

\begin{Question}{\rm Let $(R, \m)$ be a Cohen-Macaulay local ring with a canonical module $\omega_{R}$. Then the canonical module of the idealization $A=R  \ltimes \m$ is ${\ds \omega_{A} = (\omega_{R} : \m) \times \omega_{R}}$. Suppose that $\omega_{R}$ is not isomorphic to an ideal. Can the Auslander dual $D(\omega_{A})$ of $\omega_{A}$  be written in terms of the Auslander dual $D(\omega_{R})$ of $\omega_{R}$? Is it true that ${\ds \ddeg(R\ltimes \m) = 2\, \ddeg (R) -1}$?
}\end{Question}

\medskip

\section{Precanonical Ideals}\label{precanonical}

In this section we consider a Cohen-Macaulay local ring $R$ that does not necessarily have a canonical module. We look for rings that have ideals with properties similar to those of a canonical ideal. We start by recalling some useful definitions.

\smallskip

An ideal $I$ is said to be {\em closed} if $\Hom_{R}(I, I)=R$. An $R$--module $E$ is said to be {\em semidualizing} (or {\em spherical}) if ${\ds \Hom_{R}(E, E)= R}$ and ${\ds \Ext^{i}_{R}(E, E)=0}$ for all $i \geq 1$ (\cite{V74}, \cite{CSW1}). We refer to \cite{SWsemidualizing} for a survey of semidualizing modules. In particular, the canonical module is semidualizing.

\smallskip

If $R$ is not Gorenstein, principal ideals generated by a non-zero divisor are semidualizing. Examples of rings with non trivial semidualizing ideals are constructed in \cite{SW-divisorclassgroup}.

\begin{Definition}\label{precanonicalideal}{\rm
Let $(R, \m)$ be a Cohen--Macaulay local ring of dimension $d >0$. An $R$--ideal $\P$ is called a {\em precanonical} ideal if it is closed and $\Ext_{R}^{1}(\P, \P)=0$. 
}\end{Definition}

For instance, canonical ideals and semidualizing ideals are precanonical ideals. The following result shows when a precanonical ideal is canonical.

 \begin{Theorem}\label{Ulrich}
Let $(R, \m)$ be a $1$-dimensional Cohen-Macaulay local ring that has a canonical ideal.
Let $I$ be an $\m$--primary precanonical ideal of $R$. 
If $c \in I$ is such that $I/(c)\simeq (R/\m)^n$, for some $n$, then $I$ is a canonical ideal.
\end{Theorem}

\begin{proof} By \cite[Proposition 3.4]{bideg}, it is enough to show that $I$ is irreducible. Let $k=R/\m$.
Consider the exact sequence of natural maps
\[  0 \lar  (c ) \lar I	\lar I/(c) \simeq k^n \rar 0. \]
 Applying  $\Hom(\tratto, I)$ to the exact sequence, we obtain 
\[  0 \lar  R \lar c^{-1}I \lar  \left( \Ext^{1}_{R}(k, I) \right)^n  \lar 0. \]
Therefore we have
\[ \left( \Hom_{R}(k, I/cI) \right)^{n} \simeq  \left( \Ext^{1}_{R}(k, I) \right)^n \simeq c^{-1}I/R \simeq I/(c) \simeq k^{n}.\]
Thus, $ \Hom_{R}(k, I/cI)=k$.
Since $R/ I \simeq  (c)/cI $ embeds in $I/cI$, the socle of $R/I$ is $k$. Hence $I$ is irreducible.	
\end{proof}

\begin{Question}{\rm
Let $(R, \m)$ be a Cohen--Macaulay local ring possessing a precanonical ideal $\P$. Define the {\em bi-canonical degree} of $R$ {\em relative to} $\P$ as 
\[ \ddeg_{\P}(R) = \deg(\P^{**}/\P). \]
Which are the properties we can obtain from this metric to understand the structures of the ideal and the ring?
}\end{Question}

\medskip

Recall that if $\C$ is a canonical ideal of a $1$-dimensional Cohen-Macaulay local ring then $\ddeg(R)=\lambda(\C^{**}/\C)=0$ (i.e., $\C$ is reflexive) if and only if $R$ is Gorenstein (i.e., $\C$ is isomorphic to $R$). We now study reflexivity in more generality.

First, we prove a lemma inspired by \cite[Lemma 1.1]{HHS19} and \cite[Lemma 2.9]{KSyzygies} that has interest in its own right.

\begin{Lemma}\label{lemmatrace}
Let \(R\) be a \(1\)--dimensional Cohen-Macaulay local ring and \(I\) an ideal that contains a non--zero divisor.
Then \[\tr(I)^*\cong (I\otimes_R I^{-1})^* \cong  (I\otimes_R I^*)^*\cong \Hom_R(I,I^{**}).\]
\end{Lemma}

\begin{proof}
By \cite[Lemma 1.1]{HHS19}, one has \(\tr(I)=I\cdot I^{-1}.\)
Let \(K\) be the kernel of the evaluation map \( I\otimes_R I^* \longrightarrow \tr(I),\) and consider the exact sequence 
\[0\longrightarrow K\longrightarrow I\otimes_R I^*\longrightarrow \tr(I)\longrightarrow 0.\]

Let \(Q(R)\) be the total ring of fractions of \(R\), which is a flat \(R\)-module.
Since \(I\) contains a non--zero divisor, the map \[ev\otimes 1\colon (I\otimes_R I^*)\otimes_R Q(R)\longrightarrow \tr(I)\otimes_R Q(R)\]
is an isomorphism. Therefore, \(K\otimes_R Q(R)=0.\)

As \(K\) is a torsion \(R\)--module, \(\Hom_R(K,R)=0\), so dualizing the above exact sequence yields

\[0\longrightarrow \Hom_R(\tr(I), R))\longrightarrow \Hom_R(I\otimes_R I^*,R)\longrightarrow \Hom_R(K,R)=0\]
and the lemma follows.
\end{proof}

The following is an improvement of {\rm  \cite[Proposition 2.4]{BGHHV}}.
 
 \begin{Proposition}\label{resddeg3}
Let \(R\) be a \(1\)--dimensional Cohen-Macaulay local ring and \(I\) an ideal that contains a non-zero-divisor. If \(I\) is closed and reflexive,  then \(I\) is a principal ideal. 
\end{Proposition}

\begin{proof}
Let \({\ds L=\tr(I)=I\cdot I^{*}.}\)  
Then by Lemma~\ref{lemmatrace} and the fact that  \(I\) is closed and reflexive we have (realizing \(L^{*}\) as a fractional ideal in the quotient ring of \(R\))
\[ L^{*}=\Hom(L, R)=  \Hom(I \cdot I^{*}, R)= \Hom(I\otimes I^{*}, R)= \Hom(I, \Hom(I^{*}, R) ) = \Hom(I,I)  =  R.\]
But since \(R\) is \(1\)--dimensional, we have that the grade of \(L\) is less than or equal to one, so \(L\) is a free \(R\)--module. 
Hence \(I\) is an invertible ideal and as \(R\) is local we have that \(I\) is principal.
\end{proof}

\begin{Question}{\rm
Can we remove the hypothesis that $I$ contains a zero divisor in Proposition~\ref{resddeg3}?
}\end{Question}

\begin{Question}{\rm How does $\Ext_{R}^{1}(I,I)=0$ affect the structure of $I$?
}\end{Question}

\medskip

\begin{center}
{Acknowledgments}
\end{center}

The authors would like to thank Shiro Goto for helpful conversations about canonical modules and Keri Sather-Wagstaff for sharing her expertise on semidualizing ideals.

\medskip

\begin{center}
{Conflict of Interest Statement}
\end{center}

On behalf of all authors, the corresponding author states that there is no conflict of interest.

\end{document}